\documentclass{article}
\usepackage[left=2.5cm,right=2.5cm,top=2cm]{geometry}
\usepackage[utf8]{inputenc}

\usepackage{amsfonts}
\usepackage{amsmath}
\usepackage{amsthm}
\usepackage{amssymb}
\usepackage{dsfont}
\usepackage{mathrsfs}
\usepackage{hyperref}

\usepackage{graphicx}
\usepackage{subcaption}
\usepackage{xcolor}

\newcommand{\N}{\mathbb {N}}

\newcommand{\R}{\mathbb {R}}

\newcommand{\tr}{{\rm tr}}

\newcommand{\cF}{\mathcal {F}}

\newcommand{\cP}{\mathcal {P}}
\newcommand{\cN}{\mathcal {N}}

\newcommand{\cA}{\mathcal {A}}
\newcommand{\cE}{\mathcal {E}}

\newcommand{\FLOOR}[1]{{{\left\lfloor#1\right\rfloor}}} 
 
\newcommand{\ABS}[1]{{{\left| #1 \right|}}}
\newcommand{\BRA}[1]{{{\left\{#1\right\}}}}
\newcommand{\SBRA}[1]{{{\left[#1\right]}}}
\newcommand{\ANG}[1]{{{\left\langle#1\right\rangle}}}

\newcommand{\PAR}[1]{{{\left(#1\right)}}}

\newcommand{\ACO}[1]{{{\left\{\begin{array}{ll}
				#1
			\end{array}\right.}}}

\newcommand{\IND}{\mathds{1}}

\newcommand{\lapl}{\Delta}
\newcommand{\p}{\mathbb{P}}
\newcommand{\e}{\mathbb{E}}

\newcommand{\Cov}{\text{Cov}}

\DeclareMathOperator{\erf}{erf}
\newcommand{\scrP}{\mathscr{P}}
\newcommand{\tmix}{t_{\text{mix}}}

\newtheorem{thm}{Theorem}
\newtheorem{lem}{Lemma}
\newtheorem{pro}{Proposition}

\newtheorem*{cor*}{Corollary}
\theoremstyle{definition}
\newtheorem*{defi}{Definition}

\newtheorem*{remark}{Remark}

\title{Mixing profile for Glauber dynamics of the discrete Gaussian Free Field starting from super-harmonic functions}
\author{Alexandre Bristiel}

\begin{document}
	\maketitle
	\begin{abstract}
		We study the convergence rate of the heat-bath Glauber dynamics for the Discrete Gaussian Free Field on arbitrary connected finite graphs. We show that, when starting from super-harmonic initial conditions, the evolution enjoys a strong form of monotonicity. This allows us to get a sharp mixing profile as the size of the graphs diverges. More precisely, we show that mixing occurs at time $\frac{1}{2\lambda}\log(\mathcal{E})$ with window $\mathcal{O}(1/\lambda)$, where $\lambda$ is the spectral gap of the graph Laplacian and $\mathcal{E}$ is the energy of the super-harmonic initial condition. This result holds for arbitrary graphs that do not exhibit extreme connectivity properties (one way or the other). In particular, it holds for finite boxes of the grid $\mathbb{Z}^d$, in dimension $d\geq 3$.
	\end{abstract}
	\section{Introduction}
	The \textit{continuum Gaussian Free Field} (\textbf{GFF}) is a random field that is a generalization of Brownian motion to multiple dimensions of ``time". It is a random distribution (in the sense of generalized functions) from $\R^d$ to $\R$, with a Gaussian structure.  The two dimensional free field is a universal and elementary way to define a canonical random ``function". This explains why this distribution has been used as the starting point for numerous physical models \cite{glimmQuantumPhysics1987}, first in Quantum Field Theory, then in Conformal Field Theory, and Liouville Quantum Gravity, see \cite{wernerLectureNotesGaussian2021} and \cite{berestyckiGaussianFreeField2025}. Note that in physics the GFF is generally known as the Euclidean mass-less free field.
	\begin{figure}[!h]
		\centering
		\begin{subfigure}{.5\textwidth}
			\centering
			\includegraphics[width=1.1\linewidth]{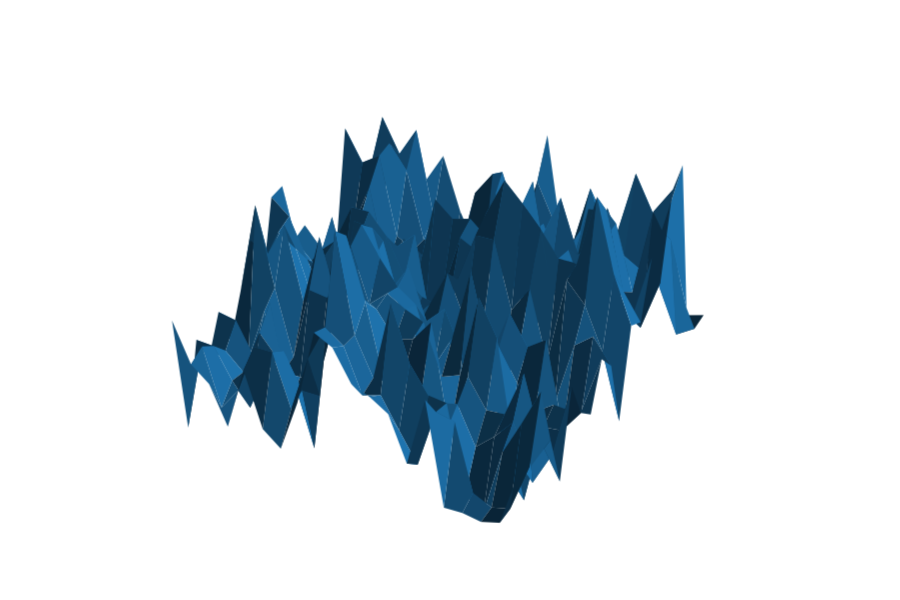}
			\caption{Continuum Gaussian Free Field.}
		\end{subfigure}%
		\begin{subfigure}{.5\textwidth}
			\centering
			\includegraphics[width=1.1\linewidth]{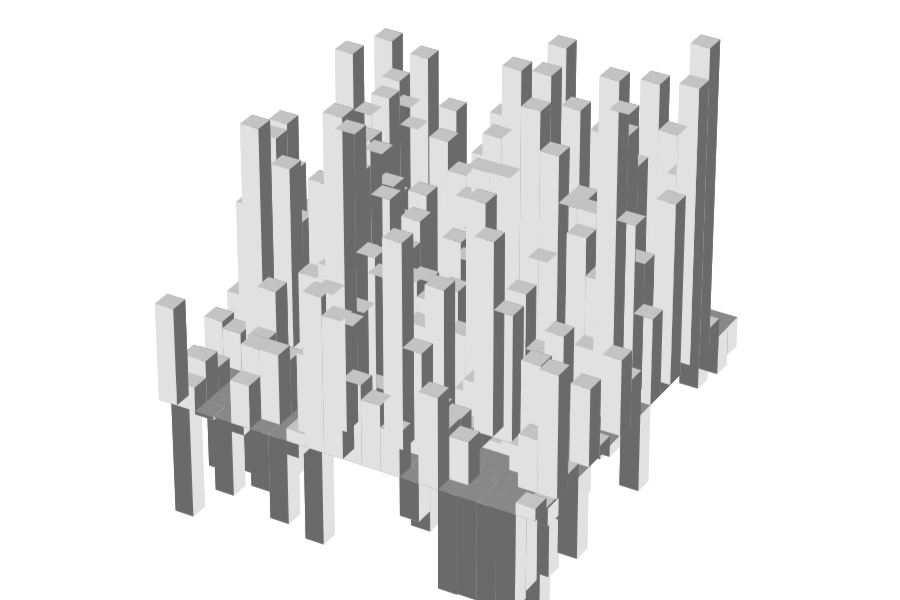}
			\caption{Discrete Gaussian Free Field.}
		\end{subfigure}
	\end{figure}
	\par{}The GFF has seen reinforced interest when a connection was made with a key random object, SLE$_4$ (\textit{Schramm-Loewner Evolution}). SLE has appeared as the scaling limit of a vast number of interface models and is the subject of intense research. It was shown in \cite{schrammContourLineContinuum2013} that one can define a zero ``level line" for the GFF and that the law of these curves is SLE$_4$. Other deep connections between both models have been found in \cite{dubedatSLEFREEFIELD2009} such as a coupling and identities for the partition function of SLE. Many more developments have been achieved since then, and the theory is evolving too quickly for a comprehensive overview.
	\par{}Just as Brownian motion arises from random walks, the GFF is expected to be the universal scaling limit of a large class of discrete random surfaces. 
	However, finding a clear description of which discrete model should converge to the GFF remains a challenge. Some important examples that have been proven to converge to the GFF include the equilibrium fluctuations of the Ginzburg–Landau $\nabla\varphi$ (GL) model \cite{millerFluctuationsGinzburgLandau2011}, the domino tiling model (dimer model) \cite{kenyonDominosGaussianFree2001} and the fluctuations of the Ginibre characteristic polynomials \cite{riderNoiseCircularLaw2007}. Understanding the dynamical evolution of the GFF is therefore of much interest, and the study of the discretization of this evolution is the focus of our paper. The canonical discretization of the GFF is the aptly named \textit{discrete Gaussian Free Field} (\textbf{DGFF}). \\
	\par{}The \textbf{DGFF} is a random real function on a discrete graph, whose correlations follow the graph structure. On finite graphs this is simply a multivariate Gaussian vector with covariance matrix the Green function of the graph. A classical fact in the study of stochastic partial differential equations is that the GFF is the stationary solution of the \textit{stochastic heat equation}:
	\[
	\frac{\partial h}{\partial t} = \lapl h + \eta,
	\]
	where $\eta$ is a space-time white noise (see for example \cite{walshIntroductionStochasticPartial1986a}). This equation induces a natural description for the dynamics of the GFF. The discretization of the stochastic heat equation describes in turn a continuous-time Markov process on the set of functions, which governs the dynamics of the DGFF. The updates done by this process changes the value of the function to the average of its neighbors plus a Gaussian noise. It turns out that this Markov process is in fact the Glaubber dynamics for the DGFF, a well know method for simulating complex high-dimensional distributions, stemming from the Ising model in \cite{glauberTimeDependentStatisticsIsing1963}.\\
	\par{}A natural question when studying such evolutions is the time they may require in order to approach equilibrium. In the context of discrete Markov chains this is known as the \textit{mixing time}. The mixing of various Glaubber dynamics has been the topic of much investigation dating from classical results for the Ising model such as \cite{aizenmanRapidConvergenceEquilibrium1987} and \cite{griffithsRelaxationTimesMetastable1966}, to the seminal works of  D. Aldous and P. Diaconis \cite{aldousShufflingCardsStopping1986, aldousRandomWalksFinite1983} and \cite{diaconisCutoffPhenomenonFinite1996}. These latter works revealed the existence of a phenomenon known as the \textit{cutoff phenomenon}, an abrupt transition from out of equilibrium to equilibrium. This discovery has led to countless developments in numerous models, although a definitive criterion for when cutoff should occur is still a major open question. \\
	\par{}This paper follows the recent developments in \cite{gangulyCutoffGlauberDynamics2023} which proved cutoff for the Glauber dynamics of the DGFF in two dimensions. We in turn consider the transient case, that is dimensions three and higher. We prove a stronger result, a full cutoff profile, but for a restricted class of initial conditions, namely super-harmonic functions. This result gives a clear understanding of the speed at which the chain mixes, provided we start with a super-harmonic initial condition. 
	
	\subsection{Discrete Gaussian free field}\label{subsec:dgff}
	Let $\overline{V} = V\cup\partial V$ be a finite state space. The set $\partial V$ is called the \textit{boundary} of $V$. Let $\overline{P} : \overline{V} \to \overline{V}$ be the transition matrix of a reversible Markov chain. We assume there are no in-place transitions, i.e. $\overline{P}$ has zero diagonal entries. Let $\overline{\lapl} = I_{\overline{V}}-P_{\overline{V}}$ be the Laplacian and $\overline{\pi}$ a invariant measure of the chain. $\overline{\pi}$ need not be a probability measure. \\
	We recall some key quantities for the space $\ell^2(\overline{\pi})$:
	\begin{itemize}
		\item For $f,g\in\ell^2(\overline{\pi})$ we denote $\ANG{\cdot,\cdot}_{\overline{\pi}}$ the \textit{scalar product}
		\[ \ANG{f,g}_{\overline{\pi}} = \sum_{x\in \overline{V}}f(x)g(x)\pi(x). \]
		\item The scalar product induced by $\overline{\lapl}$ is called the \textit{Dirichlet form} and denoted
		\[ \overline{\cE}(f,g) = \ANG{f,\lapl g}_{\overline{\pi}}. \] 
		\item The quantity $\overline{\cE}(f) := \overline{\cE}(f,f)$ is called the \textit{energy}.
	\end{itemize}
	Let $\lapl$ and $P$ be the restrictions of $\overline{\lapl}$ and $\overline{P}$ to $V$. Let $g = \lapl^{-1}$ be the \textit{Green's function}. Denote $\pi$ the conditional distribution of $\overline{\pi}$ on $V$. Let $\cE$ be the Dirichlet form induced by $\lapl$ in $\ell^2(\pi)$. Again, denote the energy $\cE(f) = \cE(f,f) = \ANG{f,\lapl f}_\pi$. Note that $\lapl, P$ and $g$ are symmetrical in $\ell^2(\pi)$ but not necessarily in the standard Euclidean space.
	\begin{defi}
		The zero-boundary \textbf{discrete Gaussian free field} (\textbf{DGFF}) $\Gamma$ is a measure on $\ell^2(\pi)$, such that
		\[ \forall f\in\ell^2(\pi),~~~\Gamma(f) \propto \exp\PAR{-\frac{1}{2}\cE(f)}. \]
		Note that $\Gamma$ is a multivariate Gaussian in the space $\ell^2(\pi)$, with mean 0 and covariance matrix $g$. The covariance matrix in the standard Euclidean space is given by $\PAR{\frac{g(x,y)}{\pi(y)}}_{x,y\in V}$.
	\end{defi}
	\begin{remark}
		The assumption of reversibility is not necessary since the distribution of $\Gamma$ depends only on the energy functional. The energy associated with $\lapl$ is the same as its symmetrized in $\ell^2(\pi)$, which is reversible.
	\end{remark}
	\begin{remark}
		The DGFF is a \textit{Gibbs measure} with Hamiltonian given by the energy $\cE$.
	\end{remark}
	\begin{remark}
		One can define the DGFF with a general boundary condition $f_{\partial V} : \partial V \to \R$, using the energy
		\[
			\tilde{\cE}(f) := \cE(f) + \sum_{x\in V, y\in \partial V}\lapl(x,y)f(x)f_{\partial V}(y)\pi(x).
		\]
		However, one can easily translate any result to this general case using the following fact. If $\gamma$ is this distributed as the DGFF with boundary condition $f_{\partial V}$ and $\overline{f_{\partial V}}$ is the harmonic extension of $f_{\partial V}$ to $V$, then $\gamma-\overline{f_{\partial V}} \sim \Gamma$.
	\end{remark}
	\subsection{Glauber dynamics}
	Our concern in this paper is the evolution of the DGFF under the \textbf{Glauber dynamics}. If $h_0 : V \to \R$ is any function, the Glauber dynamics $(h_t)_{t\geq 0}$ started from $h_0$ is a stochastic process on $\ell^2(\pi)$. \\
	The evolution can be described with the following:
	\begin{itemize}
		\item Put an exponential clock of parameter 1 at each site in $V$.
		\item If the clock at site $x$ rings at time $t$, the value of $h_t(x)$ is updated according to the law of $\Gamma$ conditioned on the remaining values, $\{h_t(y), y\neq x\}$.
	\end{itemize}
	Thanks to the Gaussianity, this conditional distribution has a simple description.
	\begin{lem}
		If $\gamma\sim \Gamma$ and $x\in V$, then the law of $\gamma(x)$ conditionally on $\{\gamma(y),y\neq x\}$ is that of a normal distribution with mean $P\gamma(x)$ and variance $1/\pi(x)$.
	\end{lem}
	\begin{figure}[!h]
		\centering
		\begin{subfigure}{.45\textwidth}
			\centering
			\includegraphics[width=.6\linewidth]{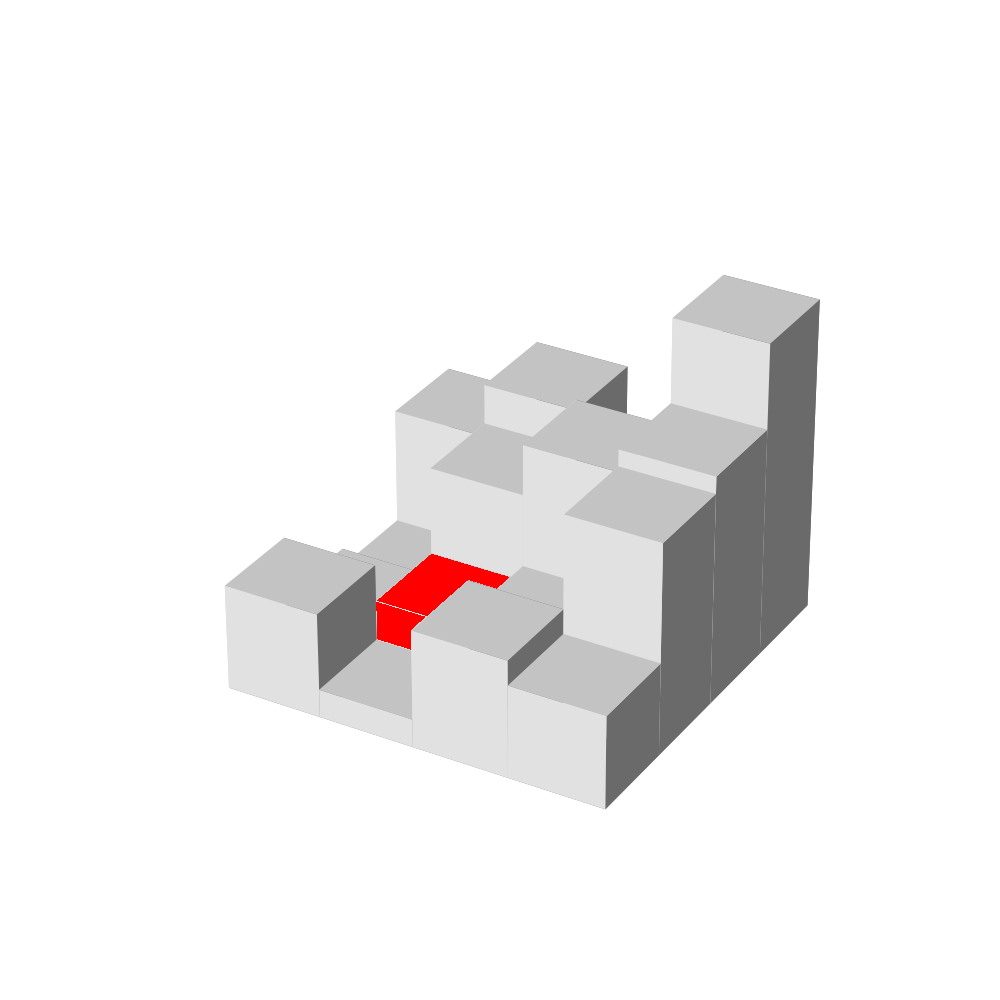}
			\caption{Before update.}
		\end{subfigure}%
		\begin{subfigure}{.45\textwidth}
			\centering
			\includegraphics[width=.6\linewidth]{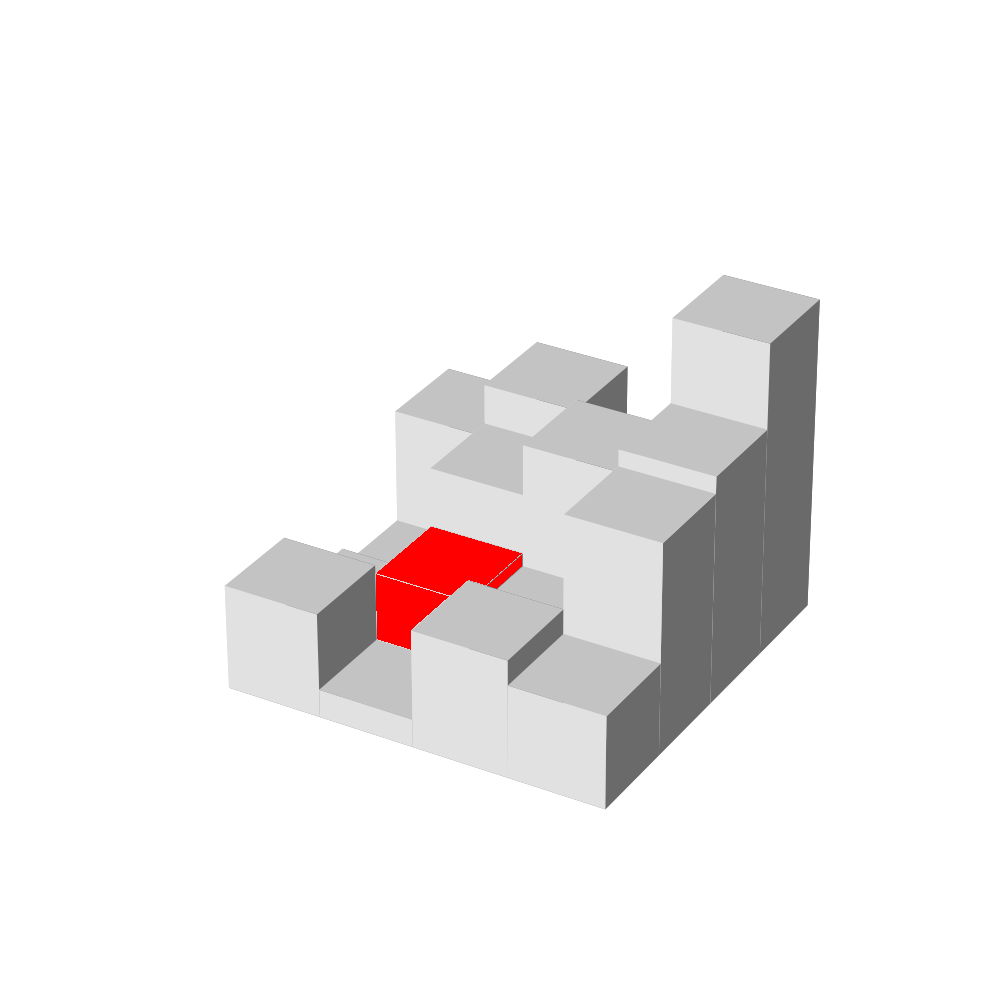}
			\caption{After update.}
		\end{subfigure}
		\caption{One site is updated to the average of its neighbors plus a Gaussian noise.}
	\end{figure}
	\begin{proof}
		We can decompose the energy using reversibility
		\begin{align*}
			\cE(f) &= \pi(x)f(x)(f(x)-Pf(x)) - \underbrace{\sum_{y\neq x}f(y)P(y,x)f(x)\pi(y)}_{=\pi(x)f(x)Pf(x)} + C\\
			&= \pi(x)(f(x) - Pf(x))^2 + C',
		\end{align*}
		where $C$ and $C'$ do not depend on the value of $f$ at site $x$.
	\end{proof}
	For any initial distribution $\mu$ on $\ell^2(\pi)$ we denote $\mu\scrP_t$ the law of the dynamics at time $t$, when initialized at $\mu$. By construction the process $(h_t)_{t\geq 0}$ is an irreducible Markov chain with invariant measure $\Gamma$ and $\delta_{h_0}\scrP_t$ will converge in law to $\Gamma$ regardless of the initial state $h_0$.
	\paragraph{Connection to the asynchronous DeGroot dynamics}If we consider only the evolution of the mean along the process, we see that at each update the value at a vertex gets replaced by the average of its neighbors. This exact process is known as the \textit{asynchronous DeGroot dynamics}. It was introduced in \cite{elboimAsynchronousDeGrootDynamics2024} as a variant of the classical DeGroot dynamics, a well-known model of non-Bayesian social learning. The only difference here is that the vertices in the boundary have fixed values.
	\begin{defi}[asynchronous DeGroot dynamics]
		Let $\overline{Q}$ be a stochastic matrix on $\overline{V}$ and $Q$ its restriction to $V$.
		The asynchronous DeGroot dynamics is a continuous time Markov chain $(\mu_t)_{t\geq 0}$ on the set of functions $V\to\R$. \\
		It has the following evolution:
		\begin{itemize}
			\item Put an exponential clock of parameter 1 at each site in $V$.
			\item If the clock at site $x$ rings at time $t$, the value of $\mu_t(x)$ is updated to $Q\mu_{t^{-}}(x)$.
		\end{itemize}
	\end{defi}
	We will describe in Subsection \ref{subsec:decomp} the exact relation between the asynchronous DeGroot dynamics and the DGFF dynamics. Our weak concentration-type result (Proposition \ref{prop:bound}) can also be interpreted as a bound on the consensus reaching time, in the context of social learning. 
	\subsection{Mixing time}
	In the following we will study the time-scale at which convergence towards equilibrium occurs. This quantity is known as the \textit{mixing time} of the chain. As is typical in the study of Markov chain mixing, we choose to measure the distance to equilibrium using the \textit{total variation distance}. 
	\begin{defi}
		The \textit{total variation distance} between two probability measures $\mu,\nu$ on the same measurable space $(\Omega,\cA)$ is given by
		\[ \|\mu-\nu\|_{TV} = \sup_{A\in\cA}|\mu(A)-\nu(A)|. \]
		An equivalent definition can be given using optimal transport duality:
		\[\|\mu-\nu\|_{TV} = \inf_{(X,Y) \text{ coupling of }\mu,\nu} \p(X\neq Y).\]
		If $\mu$ and $\nu$ both have a density with respect to a reference measure $\lambda$, then we also have
		\[
		\|\mu-\nu\|_{TV} = \frac{1}{2}\int |\mu(x)-\nu(x)|d\lambda(x).
		\]
	\end{defi}
	Since the state space is unbounded, the worst-case time for the process to converge is infinite. Thus, we need to consider the mixing time as a function of the size of the initial condition:
	\begin{defi}
		Let $\varepsilon>0$ be a precision and $r>0$ a chosen size. 
		The \textit{mixing time} at precision $\varepsilon$ from an initial condition of size $r$, denoted $\tmix^r(\varepsilon)$ is 
		\[
		\tmix^r(\varepsilon) = \max_{\substack{h : V\to\R\\ \cE(h) = r}}\inf\{t\geq 0|\|\delta_{h}\scrP_t-\Gamma\|_{TV} \leq \varepsilon\}.
		\]
		Note that $\sqrt{\cE(h)}$ is a norm on $\ell^2(\pi)$. The reason behind this choice of norm will become apparent in the proof.
	\end{defi}
 	Our objective with this paper is to show the occurrence of a phase transition for the mixing time of the chain, the \textit{cutoff phenomenon}. 

	\begin{defi}
		Let $(G_n)_{n\geq 0}$ be a sequence of connected and simple graphs with interior vertices $(V_n)_{n\geq0}$. For $n\in\N$, let $\Gamma_n$ be the DGFF on $G_n$. Assume the size of the graphs diverges. \\
		Let $r(n)$ be a real function of $n$. For $\varepsilon >0$, let $\tmix^{r,n}(\varepsilon)$ be the mixing time associated with $\Gamma_n$. We say that the dynamics of the DGFF started from an initial condition of size $r$, on the sequence of graphs $(G_n)_{n\in\N}$, exhibits a \textit{cutoff phenomenon}, if for any fixed $\varepsilon\in (0,1/2)$ we have
		\[
		\frac{\tmix^{r,n}(1-\varepsilon)}{\tmix^{r,n}(\varepsilon)} \xrightarrow{n\to\infty} 1.
		\]
		In our case the function $r$ should be chosen as a polynomial function of the size of the graph.
	\end{defi}

	\section{Main result}
	The principal obstacle to the study of this process is the complex correlations along the evolution. Thankfully we have found that when starting from a broad class of initial conditions, namely those which are \textit{super-harmonic}, we enjoy a stronger form of monotonicity at each step.
	\begin{defi}
		We say that a function $h: V\to\R$ is harmonic if $\lapl h$ is identically zero. If $\lapl h \geq 0$ (resp. $\lapl h\leq 0$) then we say that $h$ is super-harmonic (resp. $h$ is sub-harmonic). 
	\end{defi} 
	Our main result is the proof of cutoff for grids of dimension $d\geq 3$, when restricted to super-harmonic initial conditions. Note that positive constants are super-harmonic functions.
	\begin{thm}\label{thm:1}
		Let $d\geq 3$ and $n\in\N$. Let $\Lambda_n = [1,\dots, n]^d$ be the grid of length $n$ in dimension $d$. \\
		For any function $h_0 : \Lambda_n \to \R$, we denote $(h_t)_{t\geq 0}$ the Glauber dynamics for the DGFF on the grid $\Lambda_n$ with boundary equal to $\Lambda_{n+1}\setminus\Lambda_n$. By convention we take our invariant measure to be constant equal to 1 instead of re-normalizing.\\
		Let $r_n$ be any real function of $n$ such that  $ |\Lambda_n| \leq r_n \ll e^n$. Fix $s\in\R$ and set 
		\[
		t_{n} = \frac{n^2}{\pi^2}\log(r_n) + \frac{2n^2s}{\pi^2}.
		\]
		We have the following profile for the evolution started from super-harmonic functions
		\[
		\sup_{\substack{\cE(h_0) = r_n\\h_0\text{ super-harmonic}}}\|\delta_{h_0}\scrP_{t_n}-\Gamma\|_{TV} \xrightarrow{n\to\infty} \erf\PAR{\frac{e^{-s}}{2\sqrt{2}}},
		\]
		where $\erf$ is the error function 
		\[
		\erf(x) = \frac{1}{\sqrt{\pi}}\int_{-x}^xe^{-t^2}dt.
		\]
		In particular, we have cutoff at time $\frac{n^2}{\pi^2}\log(r)$ when restricting $\tmix^r(\varepsilon)$ to super-harmonic initial conditions.
	\end{thm}
	The condition that $r\geq |\Lambda_n|$ is needed in our proof to guarantee starting above the typical fluctuations of the DGFF. Note that on an arbitrary graph, if $Z\sim \Gamma$, then $\lapl^{1/2}Z$ is a standard $|V|$-dimensional Gaussian vector, so 
	\[
	\e[\cE(Z)] = \frac{1}{2}\e_\Gamma[\|\lapl^{1/2}Z\|^2] = \frac{|V|}{2}.
	\]
	\\
	Theorem \ref{thm:1} is a special case of our more general result on arbitrary discrete Gaussian Free Fields:
	\begin{thm}\label{thm:2}
		Let $(\Gamma_n)_{n\geq 0}$ be an arbitrary sequence of DGFF on sets of vertices $V_n$, as defined in Subsection \ref{subsec:dgff} (we assume the size of $V_n$ diverges).\\
		If $h_0 : V_n \to \R$ is a super-harmonic function, we denote $(\overline{h}_t)_{t\geq 0}$ the Glauber dynamics for $\Gamma_n$ started from a noisy initial distribution $\overline{h_0} = \gamma+h_0$, with $\gamma\sim\Gamma$.\\
		Let $\lambda_n$ be the spectral gap
		\[
			\lambda_n = \min_{\substack{f : V\to\R\\f\neq 0}}\frac{\ANG{f,\lapl_nf}_{\pi_n}}{\ANG{f,f}_{\pi_n}}.
		\]
		Let $r_n$ be any real function of $n$ such that
		\[
		 \max(|V_n|,\pi_n(V_n)) \leq r_n \ll \exp\PAR{\frac{1}{\sqrt{\lambda_n}}}.
		\]
		Fix $s\in\R$ and set  
		\[
		t_n = \frac{1}{2\lambda_n}\log(r_n) + \frac{s}{\lambda_n}.
		\]
		We have the following profile for the evolution started from noisy super-harmonic functions
		\[
		\sup_{\substack{\cE(h_0) = r_n\\h_0\text{ super-harmonic}}}\|\cN(h_0,g)\scrP_{t_n}-\Gamma\|_{TV} \xrightarrow{n\to\infty} \erf\PAR{\frac{e^{-s}}{2\sqrt{2}}}.
		\]
		Furthermore if $\lambda_n$ satisfies
		\[
			r_n\gg \frac{\log(1/\lambda_n)^2}{\lambda_n},
		\]
		then we also have the same profile for deterministic initial conditions
		\[
			\sup_{\substack{\cE(h_0) = r_n\\h_0\text{ super-harmonic}}}\|\delta_{h_0}\scrP_{t_n}-\Gamma\|_{TV} \xrightarrow{n\to\infty} \erf\PAR{\frac{e^{-s}}{2\sqrt{2}}}.
		\]
	\end{thm}
	For the case of grids, the spectral gap is independent of the dimension and is given by
	\[
	\lambda_n = 1-\cos\PAR{\frac{\pi}{n}} \sim \frac{\pi^2}{2n^2}.
	\]
	Since $|\Lambda_n| = n^d$, the condition $|\Lambda_n|\gg \log(1/\lambda_n)^2/\lambda_n$ is satisfied for dimensions $d\geq 3$.
	\section{Properties of the Glauber dynamics}
	We shall expose some key characteristics of this model and develop tools using these properties. First, we decompose the process in terms of its two key quantities, mean and covariance. Second, we will look at useful representation of the dynamics in terms of backward random walks as introduced in \cite{gangulyCutoffGlauberDynamics2023}. The final sections of the paper will then analyze independently each quantity, mean and covariance. Combining both estimates will complete the proof.\\
	In the following we fix $\Gamma$, an arbitrary DGFF, as in Subsection \ref{subsec:dgff}.
	
	\subsection{Decomposition}\label{subsec:decomp}
	A crucial feature of the dynamics is that, when conditioned on the clocks, the Glauber evolution is Gaussian at each fixed time. \\
	Indeed if the clock at site $x$ rings at time $t$, the function $h_t$ is updated at $x$ to be a $Ph(x)+Z/\sqrt{\pi(x)}$, where $Z$ is a standard normal variable. We can rewrite this update using $e_x$, the vector of the standard basis associated to $x$ 
	\begin{align*}
		h_t &= h_{t^-} -\ANG{h_{t^-},e_x}e_x + \ANG{Ph_{t^-},e_x}e_x + \frac{Z}{\sqrt{\pi(x)}}e_x\\
		&= h_{t^-}-\ANG{\lapl h_{t^-},e_x}e_x + \frac{Z}{\sqrt{\pi(x)}}e_x.
	\end{align*}
	From this we can deduce that at any fixed time, when conditioning on where and when the clocks ring, the evolution is a Gaussian vector.
	
	\par Since Gaussian vectors are uniquely characterized by two parameters, mean and covariance, we can get a relatively simple description of the entire evolution. In essence the process is the sum of two distinct dynamics:
	\begin{itemize}
		\item the ``noise process" $(\sigma_t)_{t\geq0}$: a zero mean Gaussian vector. Its covariance matrix monotonically increases, in the sense of symmetric positive-definite matrices, towards the Green function. Furthermore, if the initial condition is taken to be a shift of the equilibrium, this process is completely stationary.
		\item the \textit{asynchronous DeGroot dynamics} $(\mu_t)_{t\geq0}$: a deterministic averaging process, that depends only on the order and location of the clock rings. 
	\end{itemize}
	For the rest of the paper we denote $\cP$ a Poisson point process on $V\times\R_+$.  We list the elements of $\cP = \{(x_1,t_1),\dots,(x_i,t_i),\dots\}$ ordered by time $t_i$ (we assume no two $t_i$ are the same, which is true almost surely). By convention we set $t_0 = 0$.\\
	\par{}We shall detail the decomposition of the evolution $(h_t)_{t\geq 0}$ with an arbitrary initial condition $h_0 : V\to\R$. Let $Z_1,\dots,Z_i,\dots$ be i.i.d. standard normal variables, independent of $\cP$. We write $(h_t)_{t\geq0}$, as the sum of two processes,
	$(\mu_t)_{t\geq0}$ and $(\sigma_t)_{t\geq0}$. \\
	First set $\mu_0 = h_0$ and $\sigma_0 = 0$. For $i\geq 1$ and $t\in [t_{i-1},t_i[$, set:
	\begin{itemize}
		\item $\mu_t(x) = \mu_{t_{i-1}}(x)$ for $x\neq x_i$ and 
		\[
		\mu_t(x_i) = P\mu_{t_{i-1}}(x_i).
		\]
		\item $\sigma_t(x) = \sigma_{t_{i-1}}(x)$ for $x\neq x_i$ and 
		\[
		\sigma_t(x_i) = P\sigma_{t_{i-1}}(x_i) + \frac{Z_i}{\sqrt{\pi(x_i)}}.
		\]
	\end{itemize}
	By linearity we observe that $\mu_t+\sigma_t$ has the same distribution as $h_t$. We denote
	\[ \Sigma_t(x,y) =\Cov(\sigma_t(x),\sigma_t(y))\pi(y), \]
	the re-normalized covariance matrix of $\sigma_t$ (the covariance matrix in the $\ell^2(\pi)$ space).\\
	\begin{itemize}
		\item On the first hand, note that $(\sigma_t)_{t\geq0}$ is independent of the initial condition. In fact, we can see from the definition that it is the Glauber dynamics started at 0. We note here that $\Sigma_t \prec g$ in the sense of symmetric matrix ordering \cite{gangulyCutoffGlauberDynamics2023} and that the convergence of $\Sigma_t$ to $g$ is monotone.
		\item On the other hand, from the definition we see that $\mu_t$ is the asynchronous DeGroot dynamics.
	\end{itemize}
	The final parts of the paper will deal with the independent study of each quantity, $(\mu_t)_{t\geq 0}$ in Section \ref{sec:mean} and $(\Sigma_t)_{t\geq0}$ in Section \ref{sec:var}. 
	
	\subsection{Combining estimates on the two processes}\label{subsec:comb}
	In this section we shall explain how to combine estimates on $\Sigma_t$ and $\mu_t$ to obtain the main results. Consider initial conditions that are shifts of the equilibrium. For $h_0:V\to\R$, we denote $(\overline{h}_t)_{t\geq 0}$ the evolution under the Glauber dynamics of the DGFF with initial condition distributed as $\cN(h_0, g)$. Let us decompose again the evolution.\\ 
	We keep the same construction for the variance process $(\overline{\sigma}_t)_{t\geq 0}$, except that we set $\overline{\sigma_0}\sim\Gamma$. Recall that the evolution rules for $(\sigma_t)_{t\geq0}$ are exactly the Glauber dynamics, so with this new definition $\overline{\sigma_t}$ is stationary for all time $t$. In fact we have a stronger result.
	\begin{lem}
		For any time $t\geq 0$, $\overline{\sigma_t}$ is independent of the clocks $\cP$.
	\end{lem}
	By linearity we again have $\overline{h}_t = \overline{\sigma}_t + \mu_t$, with $\mu_0 = h_0$. We can now conclude that the process $(\overline{h_t})_{t\geq 0}$ is the independent sum of the mean process $(\mu_t)_{t\geq 0}$ and a stationary process equal to the equilibrium.
	\begin{proof}
		Fix $\cP = \{(x_1,t_1),\dots,(x_i,t_i),\dots\}$ a realization of the Poisson point process. Let us show that the law of $\overline{\sigma}_t$ conditioned on $\cP$ is $\Gamma$. It suffices to show that $\overline{\sigma}_{t_{i-1}} \sim\Gamma$ implies $\overline{\sigma}_{t_i}\sim\Gamma$.\\
		By linearity we know that if $\overline{\sigma}_{t_{i-1}}$ is a zero mean Gaussian then $\overline{\sigma}_{t_i}$ will also be, so we only need to check that the covariance matrices correspond. The only coordinate that changes are the one at site $x_i$. Calculating the entry in the new covariance matrix yields
		\begin{align*}
			\Cov(\sigma_{t_i}(x_i), \sigma_{t_i}(x_i)) &= \frac{1}{\pi(x_i)}\Cov(Z_i,Z_i) +\Cov(P\sigma_{t_{i-1}}(x_i), P\sigma_{t_{i-1}}(x_i))\\
			&= \frac{1}{\pi(x_i)} + \sum_{x,y\in V}P(x_i,x)\frac{g(x,y)}{\pi(y)}P(x_i,y)\\
			&= \frac{1}{\pi(x_i)} + \frac{1}{\pi(x_i)}(I-\lapl)g(I-\lapl)(x_i,x_i)\\
			&=\frac{g(x_i,x_i)}{\pi(x_i)},
		\end{align*}
		where we used that the diagonal of the Laplacian is constant equal to 1 and the reversibility of $P$. For other entries $x\neq x_i$ we have
		\[
		\Cov(\sigma_{t_i}(x), \sigma_{t_i}(x_i)) = \sum_{y\in V}\frac{g(x,y)}{\pi(y)}P(x_i,y) = \frac{gP(x,x_i)}{\pi(x_i)} = \frac{g(x,x_i)}{\pi(x_i)}-\underbrace{\frac{g\lapl(x,x_i)}{\pi(x_i)}}_{=0}.
		\]
		This shows that $\overline{\sigma}_{t_i}\sim\Gamma$.
	\end{proof}
	\par We recall the straightforward calculation of the total variation distance of two Gaussian vectors with the same covariance matrix.
	\begin{pro}\label{prop:1}
		We have
		\[
		\|\cN(\mu_1,\Sigma)-\cN(\mu_2,\Sigma)\|_{TV} = \erf\PAR{\frac{\|\Sigma^{-1/2}(\mu_1-\mu_2)\|}{2\sqrt{2}}},
		\]
		where $\|.\|$ is the standard Euclidean norm.
	\end{pro}
	\begin{proof}
		By an affine transformation, we only need to prove the proposition for $\Sigma = I_d$ and $\mu_2 = 0$. Then
		\begin{align*}
			\|\cN(\mu,I_d)-\cN(0,I_d)\|_{TV} &= \frac{1}{2(2\pi)^{k/2}}\int_{\R^d}\ABS{e^{-\frac{1}{2}\|x\|^2} - e^{-\frac{1}{2}\|x-\mu\|^2}}dx\\
			&=\frac{1}{2(2\pi)^{k/2}}\int_{\R^d}\ABS{1 - e^{-\frac{1}{2}(-2\ANG{x,\mu}+\|\mu\|^2)}}e^{-\frac{1}{2}\|x\|^2}dx\\
			&=\|\cN(\|\mu\|, 1) - \cN(0,1)\|_{TV}.
		\end{align*}
		So, it is enough to look at the 1D case. If $\mu$ and $\nu$ are two real measures with densities $f$ and $g$ with respect to the Lebesgue measure, then 
		\[
		\|\mu-\nu\|_{TV} = \int_{f(x) \geq g(x)}(f(x)-g(x))dx.
		\]
		In our case the set over which we integrate is explicitly
		\[
		\{x\in\R | e^{-\frac{1}{2}x^2} \geq e^{-\frac{1}{2}(x-\|\mu\|)^2}\} = ]-\infty,\|\mu\|/2].
		\]
		Setting $Z\sim\cN(0,1)$, we can conclude that
		\begin{align*}
			\|\cN(\|\mu\|, 1) - \cN(0,1)\|_{TV} &= \p(Z\leq \|\mu\|/2) - \p(\|\mu\|+Z\leq \|\mu\|/2)\\
			&=\p(-\|\mu\|/2\leq Z\leq \|\mu\|/2) = \erf\PAR{\frac{\|\mu\|}{2\sqrt{2}}}.
		\end{align*}
	\end{proof}
	Another estimate of total variation distance is that of two Gaussian vectors with same mean but different covariance matrix. This estimate is much more involved, and the proof is given in \cite{devroyeTotalVariationDistance2023} and the general case was solved in \cite{arbasPolynomialTimePrivate2023}. 
	\begin{thm}\label{thm:gauss}
		Let $\Sigma_1,\Sigma_2$ be two positive semi-definite matrix of same size. There exists an universal constant $C$ such that 
		\[\|\cN(0,\Sigma_1)-\cN(0,\Sigma_2)\|_{TV} \leq C \|\Sigma_1^{-1/2}\Sigma_2\Sigma_1^{-1/2}-I\|_F,\]
		where $\|.\|_F$ is the Frobenius norm $\|M\|_F = \sqrt{\tr(MM^T)}$.\\
	\end{thm}

	\par{}When conditioning on the clocks $\cP$, $\overline{h}_t$ is a multivariate Gaussian with same covariance as $\Gamma$ but mean $\mu_t$. Using the Proposition \ref{prop:1} and the fact that the total variation distance increases under conditioning we have already proven
	\[
	\|\cN(h_0,g)\scrP_t-\Gamma\|_{TV} \leq \e\SBRA{\erf\PAR{\frac{\sqrt{\cE(\mu_t)}}{2\sqrt{2}}}}.
	\]
	Since $\erf$ is linear near $0$, this indicates that the mixing is governed by the norm $\sqrt{\cE(\mu_t)}$. \\
	\par{} The next step is to control the distance between $h_t$ and $\overline{h}_t$. This distance can be bounded by a quantity that is independent of $h_0$. Indeed when conditioning on $\cP$, these processes are Gaussian vectors with same mean but different covariance, one being $\lapl^{-1}$ and the other $\Sigma_t$. Using the above Theorem \ref{thm:gauss} to bound the distance between $h_t$ and $\overline{h_t}$, we get
	\begin{align*}
		\|\delta_{h_0}\scrP_t-\Gamma\|_{TV} &\leq \|\delta_{h_0}\scrP_t-\cN(h_0,g)\scrP_t\|_{TV}+\|\cN(h_0,g)\scrP_t-\Gamma\|_{TV}\\
		&\leq C\e[\|\lapl^{1/2}\Sigma_t\lapl^{1/2}-I_V\|_F]+\e[\erf(\sqrt{\cE(\mu_t)}/2\sqrt{2})].
	\end{align*}
	From this bound we see that it suffices to study the convergence of the two quantities $(\mu_t)_{t\geq 0}$ and $(\Sigma_t)_{t\geq0}$ to show that $h_t$ is close to equilibrium.\\	
	\par{}One observation is that the first term does not depend on the initial condition and is essentially the total variation distance between the Glauber dynamics started at zero and the equilibrium. This is why we needed to consider initial conditions that are large enough to mix slower than the evolution started at zero. In Theorem \ref{thm:2} we show a general condition for cutoff for the noised process $(\bar{h}_t)_{t\geq 0}$ (starting from super-harmonic functions). For any graph that meets the condition of the theorem, to prove cutoff for the initial process $(h_t)_{t\geq 0}$ it suffices to show that the evolution started from zero mixes faster than that of large super-harmonic functions.
	\subsection{Backward random walk representation}\label{subsec:RW}
	Our final tool for the proof is a useful random walk representation for our quantities of interest $(\mu_t)_{t\geq 0}$ and $(\Sigma_t)_{t\geq 0}$. This representation is borrowed from \cite{gangulyCutoffGlauberDynamics2023} but is also known as the \textit{fragmentation process} or \textit{dual process} in opinion dynamics, see \cite{gantertAveragingProcessInfinite2024, elboimAsynchronousDeGrootDynamics2024, elboimEdgeaveragingProcessGraphs2025}. For completion's sake and more importantly to explicit our construction, we will re-explain the representation here. The main idea is to construct a random walk that is compatible with our evolution.
	\begin{defi}[Construction of the walk]
		Let $\Gamma$ be an arbitrary DGFF as in Subsection \ref{subsec:dgff}. Let $\cP = \{(x_1,t_1),\dots,(x_i,t_i),\dots\}$ be a Poisson point process on $V\times\R_+$, ordered by time. \\
		Let $(S_k)_{k\geq 0}$ be the discrete random walk on the graph with absorption at the boundary (independent of $\cP$). In other words, the transition kernel of $(S_k)_{k\geq 0}$ is given by
		\[
		\forall k\geq 0,~\forall x,y\in V,~\p(S_{k+1} = y|S_k =x)  = P(x,y),
		\]
		and $S_{k+1}$ is killed with probability $1-\sum_{y\in V}P(x,y)$.\\
		The \textit{(continuous) killed random walk with clocks} $\cP$ is denoted by $(X_t)_{t\geq 0}$. Its construction is as follows:
		\begin{itemize}
			\item Set $T_0 = 0$, and $X_{T_0} = Y_0$. The sequence $(T_i)_{i\geq 0}$ are the times at which the walk will jump.
			\item Set $T_{i+1} = \inf\BRA{t_j > T_i | x_j = X_{T_i}}$, the next time a clock rings at the location of the walk.
			\item Let $X_t$ be constant on the interval $[T_i,T_{i+1}[$ and set
			\[
			X_{T_{i+1}} := S_{i+1}.
			\]
		\end{itemize}  
		This construction ensures that $(X_t)_{t\geq 0}$ only jumps at the time and place a clock rang.
	\end{defi}
	\begin{figure}[!h]
	\centering
	\includegraphics[width=\linewidth]{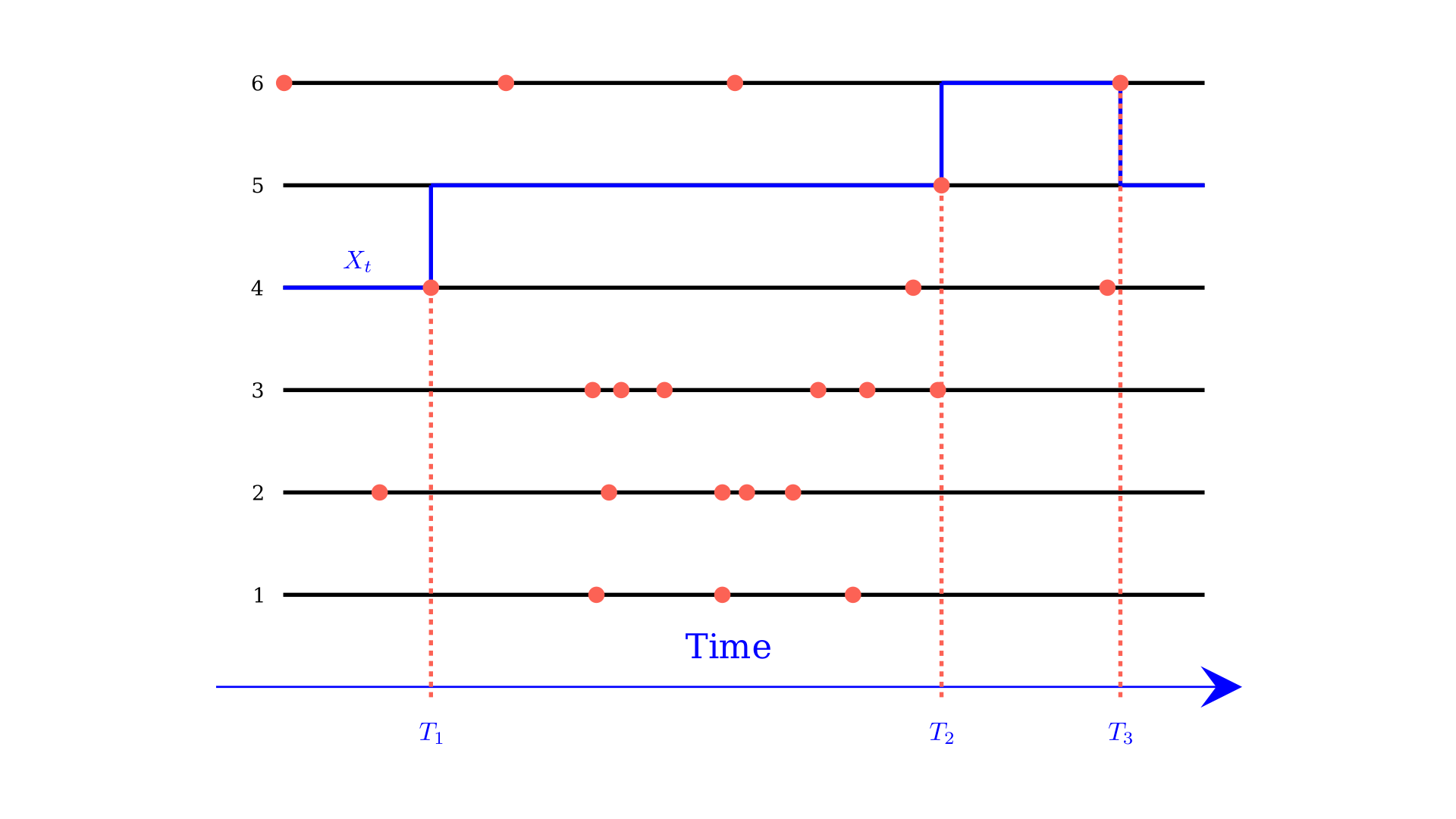}
	\caption{Random walk on 6 sites.}
	\end{figure}
	\par{}Denote $(S_k^x)_{k\geq 0}$ and $(X^x_t)_{t\geq0}$ the respective walks initialized at $x\in V$. Note that when averaging on $\cP$, $(X_t)_{t\geq0}$ is the usual continuous time random walk killed when entering $\partial V$. We also denote $\tau_\dagger$ the lifetime of $(X_t)_{t\geq 0}$.
	\begin{defi}[Backward random walk]
		If $\cP = \{(x_1,t_1), \dots, (x_i,t_i),\dots\}$. Let $\iota_t$ be the number of clocks that rang before time $t$, that is $\iota_t(\cP) =  \max\{i\geq 0 | t_i \leq t\}$. Set
		\[
		\theta_t\cP := \{(x_{\iota_t}, t-t_{\iota_t}), (x_{\iota_t-1}, t-t_{\iota_t-1}),\dots, (x_1, t- t_1)\},
		\]
		the time reversal of $\cP$.
		\par{}Note that $\theta_t$ is a measure-preserving transformation for the finite time horizon Poisson process. Since $\theta_t$ reverses the order of clocks, we call the killed random walk \textit{with clock $\theta_t\cP$} the \textit{backward random walk}.
	\end{defi} 
	\begin{pro}[Backward random walk representation]\label{prop:repr}
		Let $\Gamma$ be an arbitrary DGFF as in Subsection \ref{subsec:dgff}. Fix a time $t\geq 0$. Let $(\mu_t)_{t\geq 0}$ and $(\Sigma_t)_{t\geq 0}$ be as defined in Subsection \ref{subsec:decomp}. \\
		For any $x,y\in V$, the entries of $\mu_t$ and $\Sigma_t$ are given by
		\begin{align*}
			\mu_t(x) &= \e[\mu_0(X_t^x)\IND_{\{\tau_\dagger > t\}} | \cP = \theta_t\cP_0],\\
			\Sigma_t(x,y) &= g(x,y) - \pi(y)\e\SBRA{g(X_t^x,\tilde{X_t}^y)\IND_{\{\tau_\dagger > t\}} |\cP = \theta_t\cP_0},
		\end{align*} 
		where $(X^x,\tilde{X}^y)$ are independent continuous killed walks with the same clocks $\cP$. In particular,
		\begin{align*}
			\e[\mu_t] &= e^{-t\lapl}\mu_0,\\
			\e[\Sigma_t] &= g-  e^{-2t\lapl}g \pi^T.
		\end{align*}
	\end{pro}
	In the following, for clarity, we shall consider only forward random walks since we are only interested in average quantities and the distribution of $\theta_t\cP$ is the same as $\cP$ before any fixed time horizon $t$.\\
	\par{} This proposition indicates that if we where to believe some kind of concentration for $\mu_t$ then $\|\delta_{h}\scrP_t-\Gamma\|_{TV}$ should be up to constant equal to $e^{-\lambda t}\sqrt{\cE(h)}$, where $\lambda$ is the spectral gap of the Laplacian. This would result in a mixing time of order 
	\[
	\tmix = \frac{1}{2\lambda}\log \cE(h).
	\]
	To show such a result we need to specialize to $h_0$ super-harmonic. As we will see, this guarantees some strong monotonicity all along the evolution.
	\par{} On the other hand to deal we the covariance we first set 
	\[
	g_t(x,y) :=\pi(y)\e[g(X_t^x,\tilde{X_t}^y)\IND_{\{\tau_\dagger > t\}} |\cP = \cP_0],
	\]
	so that $\|\lapl^{1/2}\Sigma_t\lapl^{1/2}-I_V\|_F$ simplifies to $\|\lapl g_t\|_F$. This indicates that mixing started from 0 is essentially governed by the quantity $\e[\|\lapl g_t\|_F]$.
	\begin{proof}
		Let us first illustrate the graphical construction. Let $t_{i-1}$ be the last time a clock rings during our arbitrary time horizon $t_{i}$. Let us explain how to translate the changes in $\mu_{t_{i-1}}$ to $\mu_{t_i}$ in a jump for the backward random walk. \\
		We first consider any site $x$ that is not the location of the last clock (here $x=3$). There is no change between the values of $\mu_{t_{i-1}}(x)$ and $\mu_{t_i}(x)$. Furthermore a backward random walk started from $x$ cannot jump in the time interval $[0,t_{i}-t_{i-1}]$ since no clocks have rang along its path. Here we can conclude that deterministically $\mu_{t_i}(x) = \mu_{t_{i-1}}(X^x_{t_i-t_{i-1}})$.
		\begin{figure}[!h]
			\centering
			\includegraphics[width=0.9\textwidth]{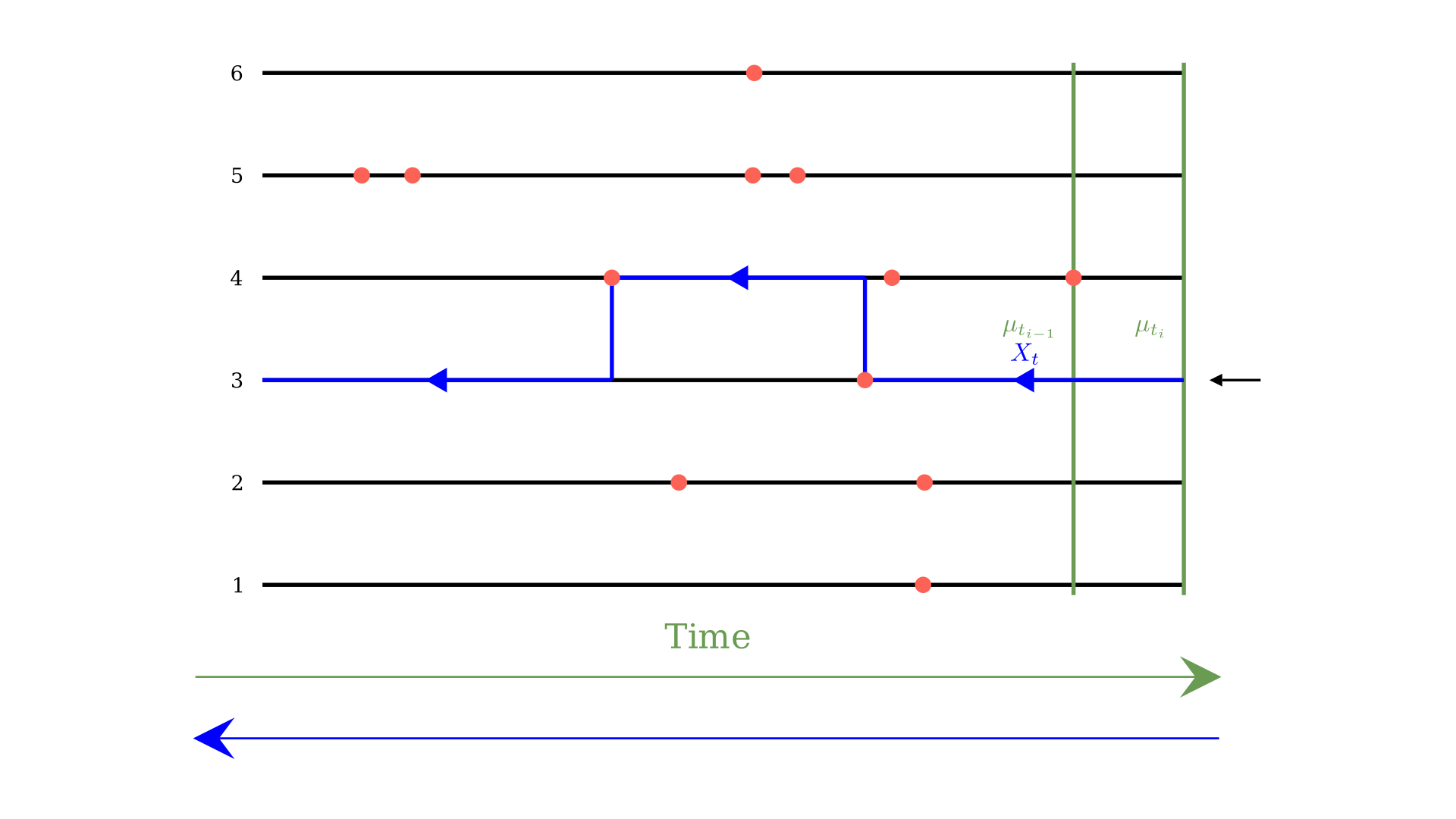}
		\end{figure}\newpage
		Next, we choose the site $x$ that is exactly the location of the last clock (here $x=4$). In this case the value of $\mu_{t_{i}}(x)$ is the average of $\mu_{t_{i-1}}$ over the neighborhood of $x$. The walk cannot jump in the time interval $[0,t_{i}-t_{i-1}[$, but will jump exactly at time $t_{i-1}-t_{i-1}$. Since the walk jumps randomly to a neighbor of $x$ we still have $\mu_{t_i}(x) = \e[\mu_{t_{i-1}}(X^x_{(t_i-t_{i-1})^+})]$. Continuing the recursion we can argue that $\mu_{t_i}(x) = \e[\mu_0(X_{t_i}^x)]$.
		\begin{figure}[!h]
			\centering
			\includegraphics[width=0.9\textwidth]{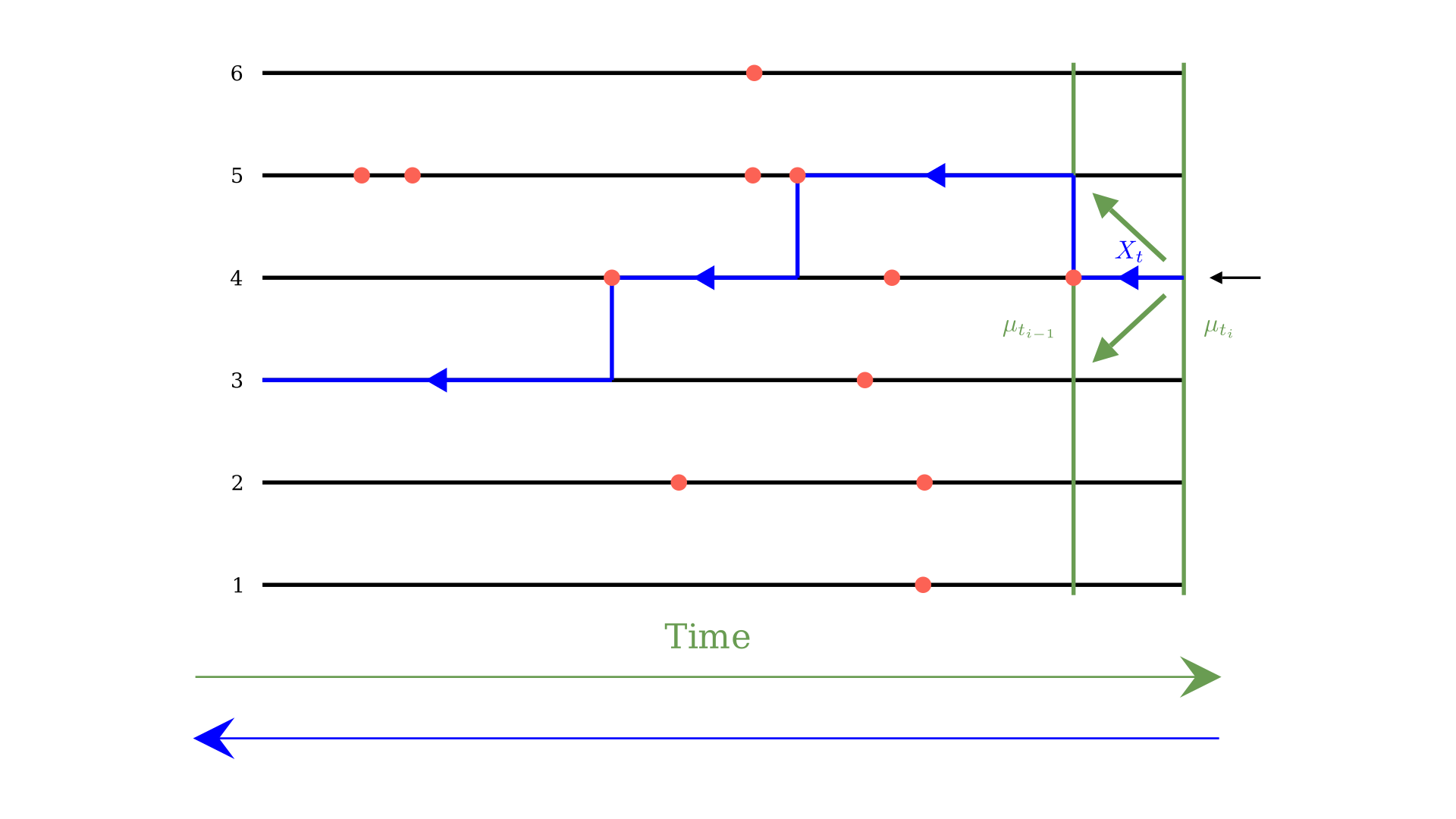}
		\end{figure}
		\par{}For the formal proof we use a recursion on $\iota_t$, the total number of clocks that rang before our time horizon $t$. If $\iota_t = 0$ then $\mu,G$ and $X$ are all constant, so the proposition is true.
		\paragraph{Mean} Let $i = \iota_t$, such that $\mu_t = \mu_{t_i}$. Take $x\in V$. \\
		First if $x\neq x_i$, the entry is unchanged, so $\mu_{t_i}(x) = \mu_{t_{i-1}}(x)$. As explained above, the key feature of this backward transformation is that the walk started at $x$ will not jump on the time interval $[0,t-t_{i-1}[$. Thus, if we remove the last point $(x_i,t_i)$ from $\cP_0$ both $\mu_t(x)$ and $\e[\mu_0(X_t^x)\IND_{\{\tau_\dagger > t\}} | \cP = \theta_t\cP_0]$ will remain unchanged. \\
		Using the inductive hypothesis on $\mu_{t_{i-1}}(x)$, we have
		\begin{align*}
			\mu_{t_i}(x) = \mu_{t_{i-1}}(x) &= \e[\mu_0(X_t^x)\IND_{\{\tau_\dagger > t\}} | \cP = \theta_t(\cP_0\setminus\{(x_i,t_i)\})] \\
			&= \e[\mu_0(X_t^x)\IND_{\{\tau_\dagger > t\}} | \cP = \theta_t\cP_0].
		\end{align*}
		If $x = x_i$ then the entry is an average of its neighbors, that is
		\[
		\mu_{t_i}(x_i) =  P\mu_{t_{i-1}}(x_i).
		\]
		On the other hand $X^{x_i}$ will jump at time $t-t_i$ at some location $y\neq x_i$ according to $P(x,\cdot)$. However any other walks started at $y\neq x_i$ do not jump before time $t-t_i$. They are therefore unchanged if we remove $(x_i,t_i)$ from $\cP_0$. In conclusion 
		\[
		\e[\mu_0(X_t^{x_i})\IND_{\{\tau_\dagger > t\}} | \cP = \theta_t\cP_0]
		= \sum_{y\in V}P(x_i, y)\e[\mu_0(X_t^y)\IND_{\{\tau_\dagger > t\}} | \cP = \theta_t(\cP_0\setminus\{(x_i,t_i)\})].
		\]
		We can conclude by using the inductive hypothesis on $\cP_0\setminus\{(x_i,t_i)\}$.
		\paragraph{Covariance matrix}Again set $\iota_t = i$. Only the $x_i$-th row/column of the covariance matrix of $\sigma_{t_i}$ differs from $\sigma_{t_{i-1}}$. The idea to deal with the other entries of $\Sigma_{t_i}$ is the same as what was presented for the mean.\\
		Consider the $x_i$-th row of $\Sigma_{t_i}$. Recall the reversibility formula
		\[ \forall x,y\in V,~ \pi(x)P(x,y) = \pi(y)P(y,x). \] 
		We will be using this formula extensively in the following.
		\begin{itemize}
			\item First, for the diagonal we have
			\begin{align*}
				\Sigma_{t_i}(x_i,x_i) &= \frac{\pi(x_i)}{\pi(x_i)} +\pi(x_i)\Cov(P\sigma_{t_{i-1}}(x_i), P\sigma_{t_{i-1}}(x_i))\\
				&=1+\sum_{x,y\in V}P(x_i,x)\Sigma_{t_{i-1}}(x,y)\frac{\pi(x_i)}{\pi(y)}P(x_i, y)\\
				&= 1 + [P\Sigma_{t_{i-1}}P](x_i,x_i).
			\end{align*}
			By definition, both $X^{x_i}$ and $\tilde{X}^{x_i}$ will jump independently at time $t_i$. Using the inductive hypothesis on $\cP_0\setminus\{(x_i,t_i)\}$, we have
			\begin{align*}
				-\e[g(X_t^{x_i},\tilde{X_t}^{x_i})\IND_{\{\tau_\dagger > t\}} |\cP = \theta_t\cP_0] &= -\sum_{x,y\neq x_i}\e[g(X_t^x,\tilde{X_t}^y)\IND_{\{\tau_\dagger > t\}} |\cP =  \theta_t(\cP_0\setminus\{(x_i,t_i)\})]P(x_i, x)P(x_i, y)\\
				&= -\sum_{x,y\in V}P(x_i, x)[g(x,y)-\Sigma_{t_{i-1}}(x,y)]\frac{1}{\pi(y)}P(x_i, y)\\
				&=\frac{1}{\pi(x_i)}\SBRA{P\Sigma_{t_{i-1}}P(x_i,x_i)-PgP(x_i, x_i)}.
			\end{align*}
			We can conclude by expanding $PgP(x_i,x_i) = (I_V-\lapl)g(I_V-\lapl)(x_i,x_i) = g(x_i,x_i)-1$.
			\item Next, for the off-diagonal entries, take $x\neq x_i$ :
			\[
			 \pi(x)\Sigma_{t_i}(x,x_i) =\pi(x_i)\Sigma_{t_i}(x_i,x) = \pi(x)\pi(x_i)\Cov(\sigma_{t_{i-1}}(x), P\sigma_{t_{i-1}}(x_i)) = \pi(x)[\Sigma_{t_{i-1}}P](x,x_i).
			\]
			Recall that $X^{x}$ cannot jump at time $t_i$ so only $X^{x_i}$ jumps. Again using the inductive hypothesis we conclude
			\begin{align*}
				-\e[g(X_t^{x},\tilde{X_t}^{x_i})\IND_{\{\tau_\dagger > t\}} |\cP = \theta_t\cP_0] &=-\sum_{y\neq x_i}\e[g(X_t^x,\tilde{X_t}^y)\IND_{\{\tau_\dagger > t\}} |\cP =  \theta_t(\cP_0\setminus\{(x_i,t_i)\})]P(x_i,y)\\
				&= -\sum_{y\in V}[g(x,y)-\Sigma_{t_{i-1}}(x,y)]\frac{1}{\pi(y)}P(x_i, y)\\
				&=\frac{1}{\pi(x_i)}\SBRA{\Sigma_{t_{i-1}}P(x,x_i)-gP(x,x_i)}.
			\end{align*}
			We can conclude by expanding $gP(x,x_i) = g(I_V-\lapl)(x,x_i) = g(x,x_i)$.
		\end{itemize}
		In any case for all $x,y\in V$:
		\[
		\Sigma_{t_i}(x,y) = g(x,y) -\pi(y)\e[g(X_t^{x},\tilde{X_t}^{x_i})\IND_{\{\tau_\dagger > t\}} |\cP = \theta_t\cP_0],
		\]
		which finishes the induction.
	\end{proof}

	\section{Monotone evolution for super-harmonic functions}\label{sec:mean}
	In this section we shall study the first of our quantities of interest, the asynchronous DeGroot dynamics $(\mu_t)_{t\geq 0}$ defined in Subsection \ref{subsec:decomp}. Our starting point is the backward random walk representation. The complex correlations between the number of jumps the walk does and the path it takes makes the analysis particularly difficult, but in the case of super-harmonic functions we can use monotonicity to heavily simplify the question.
	
	\subsection{Super-martingales}\label{subsec:mart}
	For this subsection we fix a realization of clocks $\cP_0 := \{(x_1,t_1), \dots, (x_i,t_i),\dots\}$. Let $(X_t)_{t\geq 0}$ be the \textit{killed random walk with clocks }$\cP_0$ defined in Subsection \ref{subsec:RW}. \\
	\par{}The \textit{path} of the walk is the discrete random walk $(S_k)_{k\geq 0}$ defined by $S_k = X_{t_k}$ for all $k\geq 0$, where we recall that the $(t_k)_{k\geq 0}$ are the times a clock rings on the graph. As is the case with the continuous walk, we denote $(S_k^x)_{k\geq 0}$ the path started at $x\in V$. Note that $S$ does not jump at each step. Indeed $S_k$ only jumps if $X$ jumps at time $t_k$, which is to say $(X_{t_{k-1}},t_k) \in \cP_0$. 
	\begin{lem}\label{lem:law}
		Conditionally on $\cP=\cP_0$, $(S_k)_{k\geq 0}$ is a time in-homogeneous Markov chain. Its transition kernel is given by
		\[
		\p(S_{k+1} = z | S_k = y, S_{k-1}, \dots, S_0) = \IND_{z = y}\IND_{y\neq x_{k+1}}+P(y,z)\IND_{y=x_{k+1}}.
		\]
	\end{lem}
	\begin{proof}
		Let $y\in V$, such that $S_k = X_{t_k}= y$. If $y\neq x_{k+1}$ then $X$ will not jump at time $t_{k+1}$ so $S_{k+1} = S_k$.\\
		Else we have $y=x_{k+1}$, and by definition the walk $X$ will do one jump of the discrete walk. By the Markov property for the discrete walk, the result of the jump depends only on the current location but not the entire history. We conclude that
		\[
		\p(S_{k+1} = z | S_k = x_{k+1}, S_{k-1}, \dots, S_0) = \p(X_{t_{k+1}} = z|X_{t_k} = x_{k+1}) =  P(x_{k+1},z).
		\]
	\end{proof}
	Let $\cF = (\sigma(S_0,\dots,S_k))_{k\geq 0}$ be the natural filtration of the process $(S_k)_{k\geq0}$.
	\begin{lem}\label{lem:mart}
		Let $h : V\to\R$ be a super-harmonic function. Conditionally on $\cP = \cP_0$,
		$(h(S_k))_{k\geq 0}$ is a $\cF$-super-martingale.
	\end{lem}
	\begin{proof}
		By Lemma \ref{lem:law} we write
		\begin{align*}
			\e[h(S_{k+1}) | \cF_k] &= \sum_{z\in V}h(z)(\IND_{z = S_k}\IND_{S_k\neq x_{k+1}}+P(S_k,z)\IND_{S_k = x_{k+1}})\\
			&=h(S_k)\IND_{S_k\neq x_{k+1}}+\IND_{S_k = x_{k+1}}Ph(x_{k+1}).
		\end{align*}
		Since $Ph(x_{k+1}) \leq h(x_{k+1})$ this concludes the proof.
	\end{proof}
	Accounting for Proposition \ref{prop:repr}, our goal is to understand the function $t \mapsto \e[h(X_t)|\cP]$, which translates into understanding the sequence $(\e[h(S_k)|\cP])_{k\geq 0}$. While the path $(S_k)_{k\geq 0}$ is most of the time stationary, the destination of its jumps is independent from the time at which they occur. It is thus natural to believe that by removing the downtime in $(S_k)_{k\geq0}$ we will recover the standard discrete random walk. We will show that it is indeed the case.\\
	\par{}Set 
	\[
	n_k = \sum_{i=1}^k\IND_{S_i\neq S_{i-1}},
	\]
	the number of jumps before time $k$, and 
	\[
	\tau_n = \inf\{k\geq 0 | n_k \geq n\},
	\]
	the time it takes for $(S_k)_{k\geq0}$ to make $n$ jumps. The \textit{effective path} $(\widetilde{S}_n)_{n\geq 0}$ is defined as $\widetilde{S}_n := S_{\tau_n}$ for all $n\in\N$. The effective path is the path trimmed from its stationary entries.
	\begin{pro}\label{prop:stopping}
		For any $n\geq 0$, $\tau_n$ is a $\cF$-stopping time. Furthermore, the effective path $(\widetilde{S}_n)_{n\geq 0}$ is the discrete killed random walk on $G$ and is independent of $\cP$.
	\end{pro}
	\begin{proof}
		The fact that $\tau_n$ is a $\cF$-stopping time follows from the fact that $n_k$ is $\cF_k$ measurable for any $k\geq 0$.\\\par{}
		Let $\cP_0 := \{(x_1,t_1), \dots, (x_i,t_i),\dots\}$ be a realization of clocks. Let us prove that
		\[
		\forall n\geq 0,~ \forall y,z\in V,~~~\p(S_{\tau_{n+1}} = z| S_{\tau_{n}} = y, \cP=\cP_0) = P(x,y).
		\]
		First notice that for all indices $i\in [\tau_n, \tau_{n+1})$, we have $S_i = S_{\tau_n}$, so in general
		\[
		S_{\tau_{n+1}-1} = S_{\tau_n}.
		\]
		Since there must be a jump at time $\tau_{n+1}$, we know that $(S_{\tau_{n+1}-1}, t_{\tau_{n+1}}) \in \cP_0$. By the Lemma \ref{lem:law} we have
		\[
		\p(S_{\tau_{n+1}} = z| S_{\tau_{n}} = y, \cP=\cP_0) = \p(S_{\tau_{n+1}}= z| S_{\tau_{n+1}-1} = y, \cP=\cP_0)= P(y,z).
		\]
		Notice that the result is independent of $\cP_0$, and since $S_{\tau_0} = X_0$ is independent of $\cP_0$, this concludes the proof.
	\end{proof}
	A direct corollary is that
	\[
	\forall n\geq0,~\forall x\in V,~~~\e[h(\widetilde{S}_n^x)|\cP] = P^nh(x).
	\]
	Importantly we have shown that if $h$ is super-harmonic, the path is a super-martingale. This means that on the event that the walk has jumped at least $n$ times before step $k$, we can bound the conditional expectation $\e[h(S_k^x)|\cP]$ by the deterministic quantity $P^nh(x)$.\\
	\par{}The number of jumps that the conditional walk makes is difficult to analyze in general. However the law of $\tau_n$ when averaged over $\cP$ is simply the time a continuous-time killed random walk takes to make $n$ jumps. The number of jumps a continuous-time walk does before time $t$ is a Poisson variable of parameter Poiss$(t)$. This implies, in particular, that
	\[
	\e[{\rm P}(\tau_n > \iota_t)] = \p(\text{Poiss}(t) \leq n).
	\]
	This gives us some control over the deviations of $\tau_n$. 
	\begin{lem}[Simplified Chernoff Bounds for Poisson variables]
		Let $t > 0$ and $X\sim\text{Poiss}(t)$ be a Poisson variable. Then for all $s\in (0,t)$:
		\[
		\ACO{\p(X \leq t-s) &\leq \exp(-s^2/2t),\\
			\p(X \geq t+s) &\leq \exp(-s^2/4t).}
		\]
	\end{lem}
	\begin{proof}
		Using the standard Chernoff bound for Poisson variables yields
		\begin{align*}
			\p(X \leq t-s) &\leq (1-s/t)^{-(t-s)}e^{-s}\\
			&= \exp(-t(1-s/t)\ln(1-s/t) - s).
		\end{align*}
		We use the fact that for any $x\in (0,1)$ we have $-(1-x)\ln(1-x)\leq x-x^2/2$ to get $\p(X\leq t-s) \leq e^{-s^2/2t}$. We get the other bound similarly, using $(1+x)\ln(1+x) \geq x+x^2/4$ for $x\in (0,1)$.
	\end{proof}
	\par{} 
	
	Using Poisson concentration, we know that at time $t$, with high probability, the walk will have jumped at least $t$ times (up to a factor of $\sqrt{t}$), so we can bound $\e[h_0(X_t^x)|\cP=\cP_0]$ by $P^{t}h_0(x)$. More precisely, we have the following result.
	\begin{pro}\label{prop:bound}
		Let $(\mu_t)_{t\geq0}$ be the asynchronous DeGroot process defined above in Subsection \ref{subsec:decomp}, with initial condition $\mu_0 : V\to\R$ super-harmonic. Set $\overline\mu_t = P^{\FLOOR{t}}\mu_0$. \\
		For any $t > s > 0$ and $\varepsilon>0$, we have:
		\[
		\p\PAR{\overline\mu_{t+s+1}-\varepsilon \leq \mu_t \leq \overline\mu_{t-s}+\varepsilon }\geq 1-\frac{2|V|\|\mu_0\|_\infty}{\varepsilon}e^{-s^2/4t},
		\] 
		where $\|\mu_0\|_\infty := \max\limits_{x\in V}|\mu_0(x)|.$
	\end{pro}
	\begin{proof}
		Using the backward random walk representation detailed in Subsection \ref{subsec:RW}, we know that $\mu_t = \e[h_0(X_t^x)|\cP].$\\
		Let $k = \iota_t$ be the number of clocks that rang before time $t$, such that
		\[
		\mu_t(x) = \e[h_0(S_k^x)| \cP].
		\]
		By Lemma \ref{lem:mart}, $(h_0(S_i^x))_{i\geq 0}$ is a super-martingale. Doob’s optional sampling theorem applies, meaning that for any $n,n'\geq 0$, we have 
		\[
		\e[h_0(S_{k \vee \tau_{n'}}^x)|\cP] \leq \mu_t(x)  \leq \e[h_0(S_{k \wedge \tau_n}^x)|\cP].
		\]
		We decompose the right term according to the number of jumps before time $k$
		\[
		\e_x[h_0(S_{k \wedge \tau_n})|\cP] \leq \underbrace{\e[h_0(\widetilde{S}_n^x)]}_{=P^nh_0(x)}+\|h_0\|_\infty{\rm P}_x(\tau_n > k).
		\]
		We use the same type of bound for the left term.\\
		\par{} Set $n =\FLOOR{t-s}$ and $n' = \FLOOR{t+s+1}$. We have proven the following point-wise inequality
		\[
		\forall t\geq 0,~\forall x\in V,~~~\overline\mu_{t+s+1}(x)-\|h_0\|_\infty {\rm P}_x(\tau_{\FLOOR{t+s+1}} < \iota_t) \leq \mu_t(x) \leq \overline\mu_{t-s}(x)+\|h_0\|_\infty {\rm P}_x(\tau_{\FLOOR{t-s}} > \iota_t).
		\]
		Recall that when averaging over $\cP$, the number of jumps the walk does before time $t$ is a Poisson variable, such that
		\[
		\ACO{\e[{\rm P}_x(\tau_{\FLOOR{t-s}} > \iota_t)] &= \p(\text{Poiss}(t) \leq \FLOOR{t-s}),\\
			\e[{\rm P}_x(\tau_{\FLOOR{t+s+1}} < \iota_t)] &= \p(\text{Poiss}(t) \geq \FLOOR{t+s+1}).}
		\]
		Using our simplified Chernoff bound and Markov's inequality we have
		\[
		\p\SBRA{{\rm P}_x(\tau_{\FLOOR{t-s}} > \iota_t) \geq \varepsilon},~\p[{\rm P}_x(\tau_{\FLOOR{t+s+1}} < \iota_t)\geq\varepsilon] \leq \frac{e^{-s^2/4t}}{\varepsilon}.
		\]
		We conclude by a union bound on $x\in V$.
	\end{proof}
	\subsection{Approximation of the evolution}
	This concentration-type result provides precise control on the ``mean process" which in turn governs the mixing of the equilibrium-shifted initialized chain. 
	Indeed, from our decomposition, we know that if our initial condition is given by a shift of the equilibrium measure then the mixing is totally governed by the ``mean process" $(\mu_t)_{t\geq 0}$. \\
	We have previously shown in Subsection \ref{subsec:comb}, that if $(\overline{h}_t)_{t\geq 0}$ is the evolution started from a noisy initial condition $\cN(h_0,g)$ then 
	\[
	\|\cN(h_0,g)\scrP_t-\Gamma\|_{TV} \leq \e[\erf(\sqrt{\cE(\mu_t)}/2\sqrt{2})],
	\]
	where $(\mu_t)_{t\geq 0}$ is the mean process started from $h_0$. Since we have a precise point-wise estimate of $\mu_t$ in the case of super-harmonic initial conditions, we can show a sharp profile for the mixing of this process $(\overline{h}_t)_{t\geq 0}$. \\\par{}
	Note that to have concentration of the Poisson clocks, we need the technical assumption that $\sqrt{\lambda}\log|V| \to 0$. This occurs in many examples such as grids in any dimension. Notable exceptions are expander graphs, whose spectral gap is bounded from below.
	\begin{pro}\label{prop:conc}
		Let $\Gamma$ be an arbitrary DGFF as in Subsection \ref{subsec:dgff}. Let $h_0$ be a super-harmonic function. Let $(\overline{h_t})_{t\geq0}$ evolve following the Glauber dynamic for $\Gamma$ starting from $\cN(h_0, g)$. Let $\overline{\mu}_t = P^{\FLOOR{t}}h_0$.\\ 
		Fix any $x\in\R$, arbitrary exponent $\alpha>0$. Let $t\in\R_+$ be any time  such that 
		\[
		t\geq \frac{1}{2\lambda}\log \cE(h_0) + \frac{x}{\lambda},
		\]
		where $\lambda$ is the spectral gap. We have
		\[
		\|\cN(h_0,g)\scrP_t-\cN(\overline\mu_t,g)\|_{TV}^2 \leq C_1e^{-x}\PAR{ \frac{\max(|V|,\pi(V))^{8}}{\cE(h_0)^\alpha} +  e^{C_2\sqrt{\lambda}\log\cE(h_0)}}+\frac{1}{\max(|V|,\pi(V))},
		\]
		where $C_1,C_2$ are constants depending only on $\alpha$.
	\end{pro}
	
	\begin{proof}
		We write $\overline{h_t} = \mu_t + Z$ where $Z\sim\Gamma$. Set $s := \log \cE(h_0)\sqrt{\alpha/2\lambda}$. In the following, for a function $f : V\to \R$, we denote 
		\[
			\|f\|_\pi = \sqrt{\sum_{x\in V}\pi(x)f(x)^2}= \sqrt{\ANG{f,f}_\pi},~~~\|f\|_\infty =\max_{x\in V}|f(x)|.
		\]
		\par{}Let $A$ be the event $\{\overline{\mu}_{t+s+1}-\varepsilon\leq \mu_t \leq \overline\mu_{t-s}+\varepsilon\}$, for some $\varepsilon$ to be chosen later.  By Proposition \ref{prop:bound} 
		\[
			1-\p(A) \leq \frac{2|V|\|h_0\|_\infty}{\varepsilon}e^{-s^2/4t} \leq \frac{2|V|\|h_0\|_\infty}{\cE(h_0)^{\alpha/2}\varepsilon}.
		\]
		By conditioning over $\cP$ and using Proposition \ref{prop:1}, we have
		\begin{align*}
			\|\cN(h_0,g)\scrP_t-\cN(\overline\mu_t,\lapl^{-1})\|_{TV}^2 &\leq \e[\erf(\sqrt{\cE(\mu_t-\overline\mu_t)}/2\sqrt{2})]^2\\
			&\leq C\frac{|V|^2\|h_0\|_\infty^2}{\cE(h_0)^{\alpha}\varepsilon^2}+C\e[\IND_A\ANG{\mu_t-\overline\mu_t,\lapl(\mu_t-\overline\mu_t)}_\pi],
		\end{align*}
		for some universal constant $C>0$ (because the error function is linear at 0). \\
		\par{}For the right term we note that 
		\[
		\|h_0\|_\infty \leq \frac{1}{\sqrt{\pi(V)}}\|h_0\|_\pi \leq \frac{1}{\lambda\sqrt{\pi(V)}}\cE(h_0).
		\]
		Using Cheeger's inequality and a very crude lower bound for Cheeger's constant, we can obtain
		\[
		\lambda \geq \frac{1}{2|V|^3}.
		\]
		\par{}Next, using the classical monotonicity of the model, we have deterministically, for all time $t\geq0$, $\lapl\mu_t \geq 0,\lapl\overline{\mu}_t \geq 0$. This implies that the functions $x \mapsto\ANG{x,\lapl\overline\mu_t}_\pi$  and $x\mapsto\ANG{x,\lapl\mu_t}_\pi$ are non-decreasing. We can now expand the left term to obtain
		\begin{align*}
			\IND_A\ANG{\mu_t-\overline\mu_t,\lapl(\mu_t-\overline\mu_t)}_\pi &\leq \IND_A(\ANG{\mu_t,\lapl\mu_t}_\pi+\ANG{\overline\mu_t,\lapl\overline\mu_t}_\pi -2\ANG{\mu_t,\lapl\overline\mu_t}_\pi),\\
			&\leq \ANG{\overline\mu_{t-s} + \varepsilon\textbf{1},\lapl\mu_t}_\pi+\ANG{\overline\mu_t,\lapl\overline\mu_t}_\pi -2\ANG{\overline\mu_{t+s+1} -\varepsilon\mathbf{1},\lapl\overline\mu_t}_\pi,\\
			&\leq 5\varepsilon\ANG{\textbf{1},\lapl\overline\mu_t}_\pi +\varepsilon^2\ANG{\textbf{1},\lapl\textbf{1}}_\pi+\left(\ANG{\overline\mu_{t-s},\lapl\overline\mu_{t-s}}_\pi+\ANG{\overline\mu_t,\lapl\overline\mu_t}_\pi -2\ANG{\overline\mu_{t+s+1},\lapl\overline\mu_t}_\pi\right),\\
			&\leq 5\varepsilon\ANG{\textbf{1},\lapl\overline\mu_t}_\pi + \varepsilon^2 \pi(V) + 2\ANG{\overline\mu_{t-s}-\overline\mu_{t+s+1},\lapl\overline\mu_t}_\pi,
		\end{align*}
		where we used the monotonicity of $t\mapsto \overline\mu_t$. \\\\
		We can deal with the first term by the Cauchy-Schwarz inequality
		\[ \ANG{\textbf{1},\lapl\overline\mu_t}_\pi \leq \sqrt{\pi(V)}\|\lapl\overline\mu_t\|_\pi \leq e^{-x/2}\sqrt{\frac{\pi(V)}{\cE(h_0)}} .\]
		Fixing now $\varepsilon=1/\max(|V|,\pi(V))$ makes both $5\varepsilon\|\lapl\overline\mu_t\|_1$ and $\varepsilon^2 |V|$ be bounded by a constant multiple of $1/\max(|V|,\pi(V))$.\\
		\par{}For the last term, take $(\phi_x,\lambda_x)_{x\in V}$, an orthonormal diagonalization of the Laplacian $\lapl$ in the space $\ell^2(\pi)$. Using this eigenvector decomposition we have
		\begin{align*}
			\ANG{\overline\mu_{t-s}-\overline\mu_{t+s},\lapl\overline\mu_t}_\pi &= \sum_{x\in V}\lambda_x\ANG{h_0,\phi_x}_\pi^2(1-\lambda_x)^{2t-s}\SBRA{1-(1-\lambda_x)^{2s}}\\
			&\leq \sum_{x\in V}\lambda_x\ANG{h_0,\phi_x}_\pi^2e^{-\lambda_x(2t-s)}\SBRA{1-(1-\lambda_x)^{2s}}.
		\end{align*}
		By separating the terms in the sum above, according to whether $\lambda_x \geq (1+\alpha)\lambda$ or not, we have
		\begin{align*}
			\ANG{\overline\mu_{t-s}-\overline\mu_{t+s},\lapl\overline\mu_t}_\pi &\leq \max\BRA{e^{-\lambda \alpha(2t-s)}, 1-(1-(1+\alpha)\lambda)^{2s}}e^{-\lambda(2t-s)}\sum_{x\in V}\lambda_x\ANG{h_0,\phi_x}_\pi^2\\
			&=2\max\BRA{e^{-\lambda \alpha(2t-s)}, 1-(1-(1+\alpha)\lambda)^{2s}}e^{-\lambda(2t-s)}\cE(h_0).
		\end{align*}
		By Bernouilli's inequality we have 
		\[
		\ANG{\overline\mu_{t-s}-\overline\mu_{t+s},\lapl\overline\mu_t}\leq 2( \cE(h_0)^{-\alpha}+ 2(1+\alpha)\sqrt{\alpha\lambda}\log \cE(h_0))e^{\sqrt{\alpha\lambda}\log \cE(h_0)}e^{-x},
		\]
		for some universal constant $C'$ depending only on $\alpha$. This concludes the proof.
	\end{proof}
	
	\section{Estimates on the covariance matrix}\label{sec:var}
	We conclude the proof of Theorem \ref{thm:2} by dealing with the covariance $(\Sigma_t)_{t\geq 0}$ defined in Subsection \ref{subsec:decomp}. We show that initial noise does not affect mixing if the underlying graph has enough connectivity
	\begin{lem}
		Let $\Gamma$ be an arbitrary DGFF as in Subsection \ref{subsec:dgff}. Let $\lambda$ be its spectral gap
		\[
			\lambda = \min_{\substack{f : V\to\R\\f\neq 0}}\frac{\cE(f,f)}{\ANG{f,f}}.
		\]
		Let $x\in\R$. Let $t\in\R_+$ be any time  such that 
		\[
		t\geq \frac{1}{2\lambda}\log(\max(V,\pi(V)))  + \frac{x}{2\lambda}.
		\]
		There exists a universal constant $C$ such that
		\[
		 	\sup_{h_0 : V \to \R}\|\delta_{h_0}\scrP_t-\cN(h_0,g)\scrP_t\|_{TV} \leq Ce^{-x}\frac{\log(1/\lambda)+1}{\sqrt{\lambda \max(|V|,\pi(V))}}.
		\]
	\end{lem}
	The proof of Theorem \ref{thm:2} follows from this Lemma and Proposition \ref{prop:conc}
	\begin{proof}[Proof of Theorem \ref{thm:2}]
		Under the assumptions of the theorem, $\cE(h_0) = r \geq \max(|V|,\pi(V))$. Using Proposition \ref{prop:conc} for $t = \log(r)/2\lambda + s/\lambda$ and $\alpha$ large enough means that
		\[
		 \|\cN(h_0,g)\scrP_t-\cN(\overline\mu_t, g)\|_{TV} \to 0.
		\]
		Furthermore by Proposition \ref{prop:1}, we have
		\begin{align*}
			\|\Gamma-\cN(\overline\mu_t,g)\|_{TV} &= \erf\PAR{\frac{\sqrt{\cE(\overline{\mu}_t)}}{2\sqrt{2}}}\\
			&\leq \erf\PAR{\frac{\sqrt{e^{-2\lambda t}r}}{2\sqrt{2}}}\\
			&=\erf(e^{-s}/2\sqrt{2}).
		\end{align*}
		To conclude the first result, notice that there is an equality in the above when $h_0$ is an eigenvector of the Laplacian $\lapl$ associated with the spectral gap $\lambda$.\\
		The second result is an immediate corollary from the above, since we get under the described assumption
		\[
			\sup_{\substack{\cE(h_0) = r_n\\h_0\text{ super-harmonic}}}\|\delta_{h_0}\scrP_{t}-\cN(h_0,g)\scrP_t\|_{TV} \to0.
		\]
	\end{proof}
	\begin{proof}
		Let $P_t^\cP(x,y)$ be the transition kernel of \textit{killed random walk with clocks} $\cP$. We recall the result Subsection \ref{subsec:comb} that there exists a universal constant $C$ such that
		\[
			\|\delta_{h_0}\scrP_t-\cN(h_0,g)\scrP_t\|_{TV} \leq C\e[\|\lapl g_t\|_F],
		\]
		where
		\begin{align*}
			g_t(x,y) = g(x,y) - \Sigma_t(x,y) = \sum_{u\in V} P_t^\cP(x,u)\e[g(u, X_t^y)|\cP]\pi(y).
		\end{align*}
		Since $g\lapl = I_V$, for any fixed $x\in V$ the function $y\mapsto g(x,y)$ is super-harmonic. We can therefore apply our result on super-harmonic functions.\\
		Take $s<t$ and $\varepsilon >0$, to be fixed later. We know by Proposition \ref{prop:bound} that
		\[
		\forall u,y\in \Lambda_n,~~~ [P^{\FLOOR{t+s}+1}g](u,y) - \varepsilon \leq \e[g(u, X_t^y)|\cP] \leq [P^{\FLOOR{t-s}}g](u,y) + \varepsilon ,
		\] 
		holds with probability at least
		\[
		1-\frac{2|V|\|g\|_\infty}{\varepsilon}e^{-s^2/4t}.
		\]
		Notice that $P$ and $g$ commute (since $g= (I_V-P)^{-1}$). This means that again, for any fixed $x\in V$, the function $y\mapsto [P^{\FLOOR{t-s}}g](y,x) +\varepsilon$ is super-harmonic. By apply the Proposition \ref{prop:bound} again and following the same principle, we obtain the bound
		\begin{equation}\label{eq:boundG}
			\pi(y)P^{2\FLOOR{t+s}+2}g(x,y) - 2\pi(y)\varepsilon \leq g_t(x,y) \leq \pi(y)P^{2\FLOOR{t-s}}g(x,y) + 2\pi(y)\varepsilon,
		\end{equation}
		which holds with probability at least
		\[
		1-\frac{4|V|\|g\|_\infty}{\varepsilon}e^{-s^2/4t}.
		\]
		We can find a polynomial bound for $\|g\|_\infty$, using the same estimate as the proof of Proposition \ref{prop:conc}:
		\begin{align*}
			\|g\|_\infty&\leq \frac{1}{\sqrt{|V|\pi(V)}}\sqrt{\sum_{x,y\in V}\pi(x)g(x,y)^2}\\
			&\leq\frac{1}{\lambda\sqrt{|V|\pi(V)}}\sqrt{\sum_{x,y\in V}\pi(x)(\lapl g)(x,y)^2}\\
			&\leq2|V|^{5/2}.
		\end{align*}
		If \eqref{eq:boundG} holds, then 
		\[
			\forall x,y\in V,~~~ \ABS{\lapl g_t(x,y)} \leq \pi(y)\ABS{P^{2\FLOOR{t+s}+2}g(x,y) - P^{2\FLOOR{t-s}}g(x,y)} + \pi(y)\ABS{P^{2\FLOOR{t+s}+1}g(x,y) - P^{2\FLOOR{t-s}+1}g(x,y)}+ 4\pi(y)\varepsilon.
		\]
		In the following we set $N=\max(|V|,\pi(V))$. Fix now $\varepsilon=1/N^2$ and $s=100\log N/\sqrt{\lambda}$ such that \eqref{eq:boundG} holds with probability at least $1-N^{-10}$. On that event we then get
		\[
			\|\lapl g_t\|_F \leq \|(P^{2\FLOOR{t-s}}-P^{2\FLOOR{t+s}+2})g\pi^T\|_F + \|(P^{2\FLOOR{t-s}+1}-P^{2\FLOOR{t+s}+1})g\pi^T\|_F + 4\varepsilon\sqrt{|V|\pi(V)}.
		\]
		Using the same principle as in the proof of Proposition \ref{prop:conc}, we can write these Frobenius norms as sum of eigenvalues. For an arbitrary $\beta>0$ we separate these eigenvalues depending on if they are larger than $(1+\beta)\lambda$ or not
		\begin{align*}
			\|(P^{2\FLOOR{t-s}}-P^{2\FLOOR{t+s}+2})g\pi^T\|_F^2 &= \sum_{\mu\in{\rm Sp}\lapl}\SBRA{\frac{(1-\mu)^{2(t-s)}}{\mu}\PAR{1-(1-\mu)^{4s+2}}}^2\\
			&\leq |V|\SBRA{ \frac{e^{-2\lambda(t-s)}}{\lambda}\PAR{(1+\beta)\lambda(4s+2)+e^{-2\beta\lambda(t-s)}}}^2\\
			&\leq \frac{Ce^{-2x}}{\lambda^2N}\PAR{\beta\sqrt{\lambda}\log N+N^{-\beta}}^2,
		\end{align*}
		where $C$ is a universal constant. We take $\beta=\log(1/\sqrt{\lambda})/\log(N)$ to minimize the expression above and finally get
		\[
			\|(P^{2\FLOOR{t-s}}-P^{2\FLOOR{t+s}+2})g\pi^T\|_F^2 \leq \frac{Ce^{-4x}(\log(1/\lambda)+1)^2}{\lambda N}.
		\]
		We can conclude using the same estimate for $\|(P^{2\FLOOR{t-s}+1}-P^{2\FLOOR{t+s}+1})g\pi^T\|_F$.
	\end{proof}
	
	\bibliographystyle{plain}
	\bibliography{DGFFcutoff}
	
\end{document}